\documentclass{amsart}

\usepackage{amssymb}
\usepackage{amsmath}
\usepackage{amsfonts}
\usepackage{graphicx}
\usepackage{amsthm}
\usepackage{enumerate}
\usepackage{esint}
\usepackage[mathscr]{eucal}
\usepackage{mathrsfs}
\usepackage{mathtools}
\usepackage{verbatim}
\usepackage{physics}
\usepackage{geometry}
\usepackage{anyfontsize} %允许字体连续缩放

\usepackage[
    colorlinks=true,
    urlcolor=blue,
    citecolor=blue,
    linkcolor=red,
    linktocpage,
    pdfpagelabels,
    bookmarksnumbered,
    bookmarksopen,
    pagebackref
]{hyperref} %生成引用链接，注：该宏包可能与其他宏包冲突，故放在所有引用的宏包之后
\usepackage[capitalize]{cleveref} %实现图片和表格、公式的引用，cleveref包必须放在hyperref包之后

\renewcommand*{\backrefalt}[4]{%
    \ifcase #1\relax
    \else
        \unskip\space #2%
    \fi
}

\crefformat{equation}{#2(#1)#3} %句中单个公式
\Crefformat{equation}{#2(#1)#3} %句首单个公式

\crefmultiformat{equation} %句中
  {#2(#1)#3} % 第一个
  { and #2(#1)#3} % 恰好两个时的第二个
  {, #2(#1)#3} % 三个以上时的中间项
  {, and #2(#1)#3} % 三个以上时的最后一项

\Crefmultiformat{equation} %句首
  {#2(#1)#3} % 第一个
  { and #2(#1)#3} % 恰好两个时的第二个
  {, #2(#1)#3} % 三个以上时的中间项
  {, and #2(#1)#3} % 三个以上时的最后一项

\crefrangeformat{equation}{#3(#1)#4--#5(#2)#6} %单个连续区间
\Crefrangeformat{equation}{#3(#1)#4--#5(#2)#6} %单个连续区间

\crefrangemultiformat{equation}
  {#3(#1)#4--#5(#2)#6}
  { and #3(#1)#4--#5(#2)#6}
  {, #3(#1)#4--#5(#2)#6}
  {, and #3(#1)#4--#5(#2)#6}

\Crefrangemultiformat{equation}
  {#3(#1)#4--#5(#2)#6}
  { and #3(#1)#4--#5(#2)#6}
  {, #3(#1)#4--#5(#2)#6}
  {, and #3(#1)#4--#5(#2)#6}

\makeatletter
\@namedef{subjclassname@2010}{%
  \textup{2010} Mathematics Subject Classification}
\makeatother

\numberwithin{equation}{section}

\theoremstyle{plain}
\newtheorem{theorem}{Theorem}[section]
\newtheorem{lemma}[theorem]{Lemma}
\newtheorem{proposition}[theorem]{Proposition}
\newtheorem{corollary}[theorem]{Corollary}

\theoremstyle{definition}
\theoremstyle{plain}

\theoremstyle{remark}
\newtheorem{remark}[theorem]{Remark}

\numberwithin{figure}{section}

\DeclareMathOperator{\diag}{diag}

\DeclareMathOperator{\vv}{\mathbf{V}}
\DeclareMathOperator{\nn}{\mathbf{n}}
\DeclareMathOperator{\hh}{H}

\newcommand{\B}{\mathbb{B}}
\newcommand{\ints}{\int_{\mathbb{S}^{d-1}}}
\newcommand{\muu}{\widehat{\mu}}
\newcommand{\D}{\mathcal{D}(\Omega)}
\newcommand{\A}{\mathcal{A}(\Omega)}
\newcommand{\s}{\mathcal{S}(\Omega)}

\begin{document}
% \date{\today} 

\title{Spectral and Geometric Stability for the Reciprocal Sum of Neumann Eigenvalues}
	
	\author{Yulong Li, Zuoqin Wang}
	\thanks{Partially supported by   NSFC no. 12571064.} 
\address[Y.L.]{School of Mathematical Sciences\\
University of Science and Technology of China\\
Hefei, 230026\\ P.R. China\\} 
\email{liyulon@mail.ustc.edu.cn}
 
\address[Z.W.]{School of Mathematical Sciences\\
University of Science and Technology of China\\
Hefei, 230026\\ P.R. China\\} 
\email{wangzuoq@ustc.edu.cn}

\begin{abstract}
    We establish quantitative stability for the reciprocal-sum isoperimetric inequality for the first $d$ nonzero Neumann eigenvalues, recently proved by He, Li, and Tang. We prove that for bounded Lipschitz domains in $\mathbb{R}^d$, the reciprocal-sum deficit controls quadratically both the normalized eigenvalue displacements and the Fraenkel asymmetry, while controlling the displacement of the eigenvalue sum linearly. A further result shows that the classical Szeg\H{o}--Weinberger deficit controls both the gap between the first two nonzero eigenvalues and the full width of the first eigenvalue cluster. Nearly spherical perturbations demonstrate that all the stability exponents are optimal. In dimension two, the same matrix method yields improved constraints on the joint spectral image of the first two nonzero eigenvalues.
\end{abstract}

\maketitle
%%%%%%%%%%%%%%%%%%%%%%%%%%%%%%%%%%%%%%%%%%%%%%%%%%%%%%%%%%%%%%%%%%%%%%%%%%%%%%%%%%%%%%%%%%%%%

% content
% \tableofcontents

\section{Introduction}
\subsection{Background}
Let $d\geq 2$. For any bounded Lipschitz domain $\Omega\subset\mathbb{R}^d$, its Neumann eigenvalues are denoted by
\[
    0=\mu_0(\Omega)<\mu_1(\Omega)\leq\mu_2(\Omega)\leq\dots\nearrow +\infty.
\]
The optimization of low nonzero Neumann eigenvalues under a volume constraint is a classical problem in spectral geometry. In the planar simply connected setting, Szeg\H{o} \cite{Sze54} proved that the disk maximizes the first nonzero eigenvalue. Weinberger \cite{Wein56} subsequently removed the dimensional and topological restrictions. More precisely, if $\B$ denotes the unit ball and $|\B|=\omega_d$, the Szeg\H{o}--Weinberger inequality reads
\[
    |\Omega|^{\frac{2}{d}}\mu_1(\Omega)\leq |\B|^{\frac{2}{d}}\mu_1(\B),
\]
with equality if and only if $\Omega$ is a ball.

The first nonzero Neumann eigenvalue of a ball has multiplicity $d$. It is therefore natural to seek an isoperimetric inequality that reflects the whole first eigenspace rather than only its lowest eigenvalue. This leads to the reciprocal-sum problem
\[
    \frac{1}{|\Omega|^{\frac{2}{d}}}\sum_{i=1}^d\frac{1}{\mu_i(\Omega)}\geq \frac{1}{|\B|^{\frac{2}{d}}}\frac{d}{\mu_1(\B)}.
\]
As already observed by Szeg\H{o} and Weinberger (see \cite[p. 634]{Wein56}), Szeg\H{o}'s conformal-mapping argument also yields this reciprocal-sum inequality among simply connected domains in dimension $2$. Ashbaugh and Benguria \cite{AB93N} conjectured in 1993 that the same sharp inequality holds in every dimension. Xia and Wang \cite{XW23} later proved the corresponding result for the sum of the reciprocals of the first $d-1$ nonzero Neumann eigenvalues. More recently, He, Li, and Tang \cite{HLT26+} established the full $d$-term inequality and characterized the equality case. The key step in their proof is to apply Hersch's variational principle to the $d$-dimensional trial space transplanted from the first eigenspace of the ball, thereby reducing the problem to a matrix inequality.

To simplify the scaling, set
\begin{equation}
    \label{eq:notationmuu}
    R_\Omega=\left(\frac{|\Omega|}{\omega_d}\right)^{\frac{1}{d}},\qquad \muu_i(\Omega)=R_\Omega^2\mu_i(\Omega),
\end{equation}
and define the normalized reciprocal-sum deficit by
\begin{equation}
    \label{eq:notationD}
    \D=\sum_{i=1}^d\frac{1}{\muu_i(\Omega)}-\frac{d}{\muu_1(\B)}.
\end{equation}
Although stated for smooth domains, the result in \cite{HLT26+} remains valid, with the same proof, for bounded Lipschitz domains. With the notation above, it takes the following form.

\begin{theorem}[{\cite[Theorem 1.2]{HLT26+}}]
    \label{thm:HLT}
    Let $\Omega\subset\mathbb{R}^d$ be any bounded Lipschitz domain. Then
    \[
        \D\geq 0.
    \]
    Equality holds if and only if $\Omega$ is a ball.
\end{theorem}

This theorem characterizes the zero-deficit case, but leaves open the corresponding stability question: what spectral and geometric information can be inferred when $\D$ is small? The purpose of this paper is to address this question by establishing several quantitative stability estimates.

\subsection{Main results}
Our first result shows that the reciprocal-sum deficit controls the entire cluster formed by the first $d$ nonzero Neumann eigenvalues. In particular, a small deficit forces every eigenvalue in this cluster to be close to the corresponding eigenvalue of the ball.

\begin{theorem}
    \label{thm:recsummui}
    There exists an explicit dimensional constant $\gamma_d>0$ such that, for every bounded Lipschitz domain $\Omega\subset\mathbb{R}^d$,
    \begin{equation}
        \label{eq:inversesumSWdeficit}
        \D \geq \gamma_d \max_{1\leq i\leq d} \big(\muu_i(\Omega)-\muu_i(\B)\big)^2.
    \end{equation}
    The constant $\gamma_d$ is given in \eqref{eq:gammad}.
\end{theorem} 

As an immediate consequence, the deficit also controls the splitting of the multiple first eigenvalue of the ball.

\begin{corollary}
    \label{cor:recsumspecgap}
    For every bounded Lipschitz domain $\Omega\subset\mathbb{R}^d$,
    \begin{equation}
        \label{eq:inversesumspecgap}
        \muu_d(\Omega)-\muu_1(\Omega) \leq \frac{2}{\sqrt{\gamma_d}}\sqrt{\D}.
    \end{equation}
\end{corollary}

The behavior of the sum of this eigenvalue cluster is different: its displacement is controlled linearly rather than quadratically. Set
\[
    \s=\sum_{i=1}^d\muu_i(\Omega).
\]

\begin{theorem}
    \label{thm:sumstability}
    There exists an explicit dimensional constant $\tau_d>0$ such that for any bounded Lipschitz domain $\Omega\subset\mathbb{R}^d$ we have
    \begin{equation}
        \label{eq:sumstability}
        \D \geq \tau_d \left|\s - \mathcal{S}(\B) \right|.
    \end{equation}
    The constant $\tau_d$ is given in \eqref{eq:taud}.
\end{theorem} 

Thus the reciprocal-sum deficit detects two distinct features of the cluster. It controls the splitting of the multiple eigenvalue quadratically, while controlling the collective displacement of the cluster linearly. The latter estimate provides, for the first nonzero Neumann eigenvalue cluster, an analogue of the Dirichlet cluster estimate established in \cite{BLNP26}, with the reciprocal-sum deficit playing the role of the Faber--Krahn deficit. As a further consequence, our results show that, whenever $|c|$ is sufficiently small, $\sum_{i=1}^{d}\left(\frac{1}{\muu_i} + c \muu_i\right)$ is minimized by the ball.

We next turn to geometric stability. For a measurable set of finite positive measure, define its Fraenkel asymmetry by
\begin{equation}
    \label{eq:Fraenkel}
    \A=\inf\left\{\frac{|\Omega\Delta B|}{|\Omega|}:~B\text{ is a ball and }|B|=|\Omega|\right\}.
\end{equation}
Nadirashvili \cite{Nad97} obtained a sharp stability estimate for the two-dimensional reciprocal-sum inequality among simply connected domains. Brasco and Pratelli \cite{BP12} established the sharp quantitative stability of the Szeg\H{o}--Weinberger inequality; we refer to \cite{BD17} for a broader account of quantitative spectral inequalities. For the reciprocal-sum deficit, we prove the following estimate in every dimension.

\begin{theorem}
    \label{thm:recsumfraenkel}
    There exists an explicit dimensional constant $\kappa_d>0$ such that, for every bounded Lipschitz domain $\Omega\subset\mathbb{R}^d$,
    \begin{equation}
        \label{eq:FraenkelStability}
        \D \geq \kappa_d\A^2.
    \end{equation}
    The constant $\kappa_d$ is given in \eqref{eq:kappad}.
\end{theorem}

%\begin{remark}
%    In dimension two, $\kappa_2\approx4.6001\times10^{-3}$.
%\end{remark}

We next turn to our main spectral-gap estimates, which compare the splitting of the first eigenvalue cluster directly with the classical Szeg\H{o}--Weinberger deficit. The next two theorems control, respectively, the gap between the first two nonzero eigenvalues and the full width of the cluster.

\begin{theorem}
    \label{thm:mu2}
    Let $\Omega\subset\mathbb{R}^d$ be a bounded Lipschitz domain. Then
    \begin{equation}\label{eq:mu2-mu1}
        \muu_2(\Omega)-\muu_1(\Omega) \leq C_1 \big(\muu_1(\B)-\muu_1(\Omega)\big), 
    \end{equation}
    where $C_1<\frac{2^{\frac{2}{d}}d}{d-1}$ is the larger root of \cref{eq:C4}.
\end{theorem}
As a consequence,  for sufficiently small $c>0$, the functional $(1-c)\muu_1+c\muu_2$ is maximized by the ball, with the admissible range of $c$ determined by the constant in the theorem. 

\begin{remark}
    A family of dumbbell-shaped domains chosen as in \cite[Section 5.2]{AC04} shows that the optimal coefficient in \cref{thm:mu2} is at least $2^{\frac{2}{d}}$. In dimension two, a numerical computation for the union of two unit disks whose centers are approximately $1.64$ apart gives the ratio $2.452$, suggesting a stronger lower bound. On the other hand, \cref{prop:2dimmu2upper} improves the upper bound $4$ furnished by \cref{thm:mu2} to $\frac{4\muu_1(\B)+2}{\muu_1(\B)+1}\approx3.544$.
\end{remark}

\begin{theorem}
    \label{thm:mud}
    Let $d\geq3$, and let $\Omega\subset\mathbb{R}^d$ be a bounded Lipschitz domain. Then
    \begin{equation}
        \label{eq:mud-mu1}
        \muu_d(\Omega)-\muu_1(\Omega)\leq C_2\big(\muu_1(\B)-\muu_1(\Omega)\big),
    \end{equation}
    where   $C_2$ is the larger root of \cref{eq:C5}.
\end{theorem}

The exponents in the preceding spectral and geometric estimates cannot be improved.

\begin{theorem}
    \label{thm:sharpness}
    The exponents appearing in the stability estimates \cref{eq:inversesumSWdeficit,eq:inversesumspecgap,eq:sumstability,eq:FraenkelStability,eq:mu2-mu1,eq:mud-mu1}     are sharp.
\end{theorem}

% We remark that the spectral and geometric estimates above are complementary. A concentric spherical shell with sufficiently small inner radius has positive Fraenkel asymmetry, whereas its first nonzero Neumann eigenvalue has multiplicity at least $d$, so the corresponding spectral splitting vanishes. Conversely, let $\Omega_\varepsilon$ be obtained from the unit ball by removing an $\varepsilon$-neighborhood of an equatorial hyperplane, except for a central passage of radius $\varepsilon$. Each $\Omega_\varepsilon$ is a connected Lipschitz domain, and $\mathcal{A}(\Omega_\varepsilon)\to 0$ as $\varepsilon\to 0^+$. As the central passage collapses, $\Omega_\varepsilon$ converges spectrally for the Neumann Laplacian, to the disjoint union of two half-balls. Consequently, $\mu_1(\Omega_\varepsilon)\to 0$, whereas $\mu_i(\Omega_\varepsilon)$ ($2\leq i\leq d$) converges to the first positive Neumann eigenvalue of a half-ball. Thus, none of the spectral quantities appearing in our stability estimates can be controlled by the Fraenkel asymmetry on the class of general Lipschitz domains.

Finally, in dimension two we derive improved constraints on the joint spectral image of the first two nonzero Neumann eigenvalues. According to  the Bucur--Henrot inequality \cite{BH19} and  \cref{thm:HLT},  the first two nonzero Neumann eigenvalues
 $\muu_1$ and $\muu_2$ satisfy the inequalities 
\[
  \muu_2(\Omega) \le 2\muu_1(\mathbb B), \quad \frac{1}{\muu_1(\Omega)}+\frac{1}{\muu_2(\Omega)} \ge \frac{2}{\muu_1(\mathbb B)} \quad \text{and} \quad \muu_1(\Omega) \le \muu_2(\Omega).
\]
Consequently, the  $(\muu_1,\muu_2)$-image lies in the following region in the first quadrant bounded by the lines $\muu_1=\muu_2$, $\muu_2=2\muu_1(\mathbb B)$, and the hyperbola $\frac{1}{\muu_1}+\frac{1}{\muu_2}= \frac{2}{\muu_1(\mathbb B)}$, as shown in Figure \ref{fig:mu1mu2-image}. The improved inequalities \cref{eq:inversesumSWdeficit} and \cref{eq:sumstability} yield only a slight refinement of this region because 
\[\tau_2\approx9.7649\times10^{-4},\qquad   \gamma_2\approx1.9204\times10^{-4}.\]
To achieve a visible improvement, we employ the full information contained in the two-dimensional matrix  comparison which yields an additional bilinear constraint and thus excludes a further portion of the previously admissible region.

\begin{proposition}
	\label{prop:2dimmu2upper}
	Let $\Omega\subset\mathbb{R}^2$ be a bounded Lipschitz domain. Then
    \[
        \muu_2(\Omega)\leq     
        \begin{cases}
            2\muu_1(\B), & 0<\muu_1(\Omega)\leq \dfrac{2\muu_1(\B)^2}{3\muu_1(\B)+1},\\
            \dfrac{(\muu_1(\B)-1)\muu_1(\Omega)+2\muu_1(\B)} {2\muu_1(\Omega)+1-\muu_1(\B)}, & \dfrac{2\muu_1(\B)^2}{3\muu_1(\B)+1} \leq \muu_1(\Omega)\leq\muu_1(\B).
        \end{cases}
    \]
	Equality holds if and only if $\Omega$ is a disk.
\end{proposition}
\vspace{-.3cm}

\begin{figure}[h]
	\centering
	\includegraphics[width=0.71\textwidth]{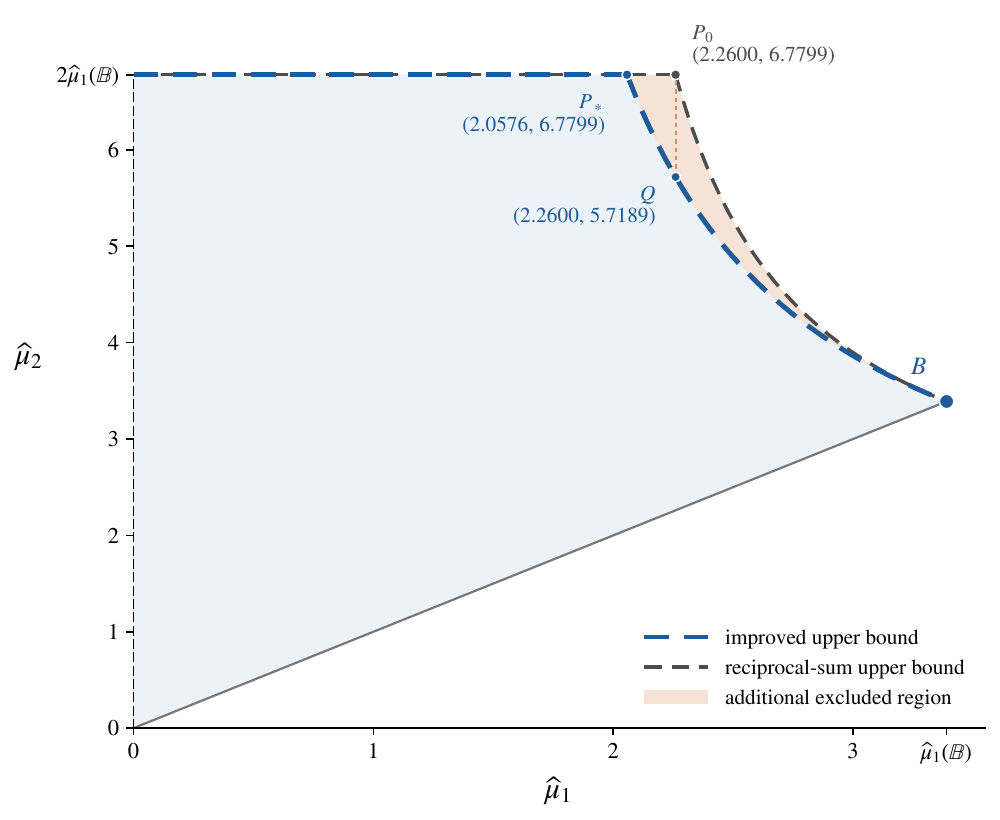}
	\caption{Constraints on the $(\muu_1,\muu_2)$-image. The gray dashed curve is obtained from the reciprocal-sum inequality together with the Bucur--Henrot bound. % whereas the blue dashed curve represents the improved constraint in \cref{prop:2dimmu2upper}. %T
		The orange region is the additional region excluded by \cref{prop:2dimmu2upper}.%, and the light-blue region is the part not excluded by the resulting constraints, bounded below by $y=x$.
	}
	\label{fig:mu1mu2-image}
\end{figure}

\subsection{Proof strategy}
The common starting point is Weinberger's family of $d$ trial functions transplanted from the first nonzero Neumann eigenspace of the ball. Let $A$ and $B$ be the associated mass and energy matrices, defined in \cref{eq:defAB}. Following He, Li, and Tang \cite{HLT26+}, we keep the trial functions coupled and apply Hersch's variational principle to the resulting finite-dimensional trial space, while Poincaré's variational principle compares the Neumann eigenvalues with the Ritz values. The comparison of the two matrices is governed by a symmetric trace-free matrix measuring the angular imbalance of the domain. Our key quantitative matrix lemma shows that the reciprocal-sum deficit controls this imbalance. Combined with the variational comparison, this yields the quadratic estimates for the individual eigenvalue displacements and the spectral splitting. The trace-free structure also produces a cancellation upon summation, leading to the linear estimate for the displacement of the eigenvalue sum.

For the geometric stability estimate, we combine the quantitative matrix lemma with the positive-semidefinite remainder retained in the energy comparison. When the angular imbalance is large, the matrix lemma already gives the required bound. When it is small, the remainder controls a weighted measure of $\B\backslash\Omega$, which is bounded below by the square of the Fraenkel asymmetry through a mass-transfer argument.

The sharpness arguments are based on first-order volume-preserving perturbations of the ball and analytic perturbation theory for its multiple first nonzero Neumann eigenvalue. A suitable degree-$2$ deformation produces a linear splitting of the eigenvalue cluster, whereas the reciprocal-sum deficit remains of quadratic order. For the remaining estimates, we choose a deformation orthogonal to all spherical harmonics of degrees at most $2$, so that the entire cluster is stationary to first order. The asymmetry estimate for nearly spherical domains in \cite[Lemma 6.2]{BDR12}, combined with our geometric stability estimate, then shows that the reciprocal-sum deficit is of quadratic order. Expanding the reciprocal sum further shows that the displacement of the eigenvalue sum has the same order.

\subsection{Organization of the paper}

The remainder of the paper is organized as follows. In Section~2, we recall the variational setup and establish the quantitative matrix lemma underlying all the stability estimates. In Section~3, we prove the spectral stability results \cref{thm:recsummui,cor:recsumspecgap,thm:sumstability}. Section~4 is devoted to the geometric stability estimate \cref{thm:recsumfraenkel}. In Section~5, we prove the spectral-gap estimates \cref{thm:mu2,thm:mud} in terms of the Szeg\H{o}--Weinberger deficit. In Section~6, we establish the sharpness of the exponents by constructing suitable first-order volume-preserving perturbations of the ball. Finally, in Section~7, we specialize the matrix argument to dimension two and derive the improved constraints on the $(\muu_1,\muu_2)$-image stated in \cref{prop:2dimmu2upper}.

\section{Preliminaries}

Let $\B\subset\mathbb{R}^{d}$ be the unit ball. The eigenspace corresponding to $\mu_1(\B)=\dots=\mu_d(\B)$ is spanned by
\[
    r^{1-\nu}J_\nu(p_{\nu,1}r)\frac{x_i}{r} =:g(r)\frac{x_i}{r}, \qquad i=1,\dots,d,
\]
where $\nu=\frac{d}{2}$, $r=|x|$, and $p_{\nu,1}$ is the first positive zero of the derivative of the function $r\mapsto r^{1-\nu}J_\nu(r)$. Let $\Omega\subset\mathbb{R}^d$ be a bounded Lipschitz domain with volume $\omega_d$. Then $\muu_i(\Omega)=\mu_i(\Omega)$. Define
\[
    G(r)=g(r)1_{[0,1]}+g(1)1_{(1,\infty)}, \quad u_i=G(r)\frac{x_i}{r}.
\]
Using the Brouwer fixed-point theorem, Weinberger \cite{Wein56} proved that, after a suitable translation, one has $\int_\Omega u_i\,\dd x=0$ for $i=1,\dots,d$. Set
\begin{equation}
    \label{eq:defAB}
    A_{ij}= \int_\Omega u_i u_j \dd x, \quad B_{ij}= \int_\Omega \nabla u_i\cdot\nabla u_j \dd x.
\end{equation}
Both $A=(A_{ij})$ and $B=(B_{ij})$ are symmetric and positive definite. Since $A^{-1}B$ is similar to $A^{-\frac12}BA^{-\frac12}$, all its eigenvalues are positive. We denote them by
\[
    p_1\leq\dots\leq p_d.
\]
By Poincaré's variational principle \cite[pp. 259-261]{Poi90}\cite[Chapter II]{PS54}, the Ritz values $p_i$ satisfy
\begin{equation}
    \label{eq:mup}
    \mu_i(\Omega)\leq p_i,
    \qquad i=1,\ldots,d.
\end{equation}
Consequently,
\begin{equation}
    \label{eq:recsump}
    \sum_{i=1}^d \frac{1}{\mu_i(\Omega)}\geq \sum_{i=1}^d \frac{1}{p_i} = \Tr (AB^{-1}) = \Tr (B^{-1}A).
\end{equation}
This inequality is also known as Hersch's variational principle \cite{Her61,HX93}. We remark that both  \cref{eq:recsump} and \cref{eq:mup} play    important roles  in our stability argument.

Let $\theta=\frac{x}{r}$. Then
\begin{equation}
    \label{eq:Aexp}
    A = \int_{\Omega} G^2\theta\theta^T \dd x = \frac{\int_{\B}G^2 \dd x}{d}I + \left(\int_{\Omega\backslash \B}-\int_{\B\backslash\Omega}\right) G^2\theta\theta^T \dd x
\end{equation}
and 
\begin{equation}
    \label{eq:Bexp}
\begin{aligned}
    B & = \int_\Omega [G'^2\theta\theta^T + \frac{G^2}{r^2}(I-\theta\theta^T)] \dd x\\
    & = \mu_1(\B)\frac{\int_{\B}G^2 \dd x}{d}I + \left(\int_{\Omega\backslash \B}-\int_{\B\backslash\Omega}\right)[G'^2\theta\theta^T + \frac{G^2}{r^2}(I-\theta\theta^T)] \dd x.
\end{aligned}
\end{equation}
Set
\begin{equation}
    \label{eq:defZ}
    Z = \int_\Omega \theta\theta^T \dd x - \int_{\B} \theta\theta^T \dd x = \left(\int_{\Omega\backslash\B}-\int_{\B\backslash\Omega}\right) \theta\theta^T \dd x.
\end{equation}
The matrix $Z$ is symmetric and has trace $\Tr Z = |\Omega|-|\B| = 0$.

For the reader’s convenience, we first sketch the proof of \cref{thm:HLT} given in {\cite{HLT26+}}, which relies on the following linear-algebraic lemma.
\begin{lemma}[{\cite[Lemma 2.2]{HLT26+}}]
    \label{lem:linalg}
    Let $\widehat{Z}$ be a real symmetric matrix of order $d$ with trace $0$. Suppose that the real positive-definite matrices $A$ and $B$ satisfy
    \[
        A\geq  aI+ \widehat{Z},\quad B\leq \lambda aI- \widehat{Z},
    \]
    where $a>0$  and $\lambda>0$. Then
    \begin{equation*}
        \Tr(B^{-1}A) \geq \frac{d}{\lambda},
    \end{equation*}
   with equality if and only if $\widehat{Z}=0$, $A=aI$, and $B=\lambda aI$.
\end{lemma}

\begin{proof}[Sketch of proof of \cref{thm:HLT}]
An analysis of the relevant Bessel functions shows that $G$ is nondecreasing on $(0,\infty)$ and strictly increasing on $(0,1)$. Hence, by the positive semidefiniteness of $\theta\theta^T$,
\begin{equation}
    \label{eq:AM}
    A\geq \frac{\int_{\B}G^2 \dd x}{d}I + G(1)^2Z.
\end{equation}
Another computation shows that $G'$ is strictly decreasing on $(0,1)$ and, on $[1,\infty)$, is identically $0$, while $\frac{G}{r}$ is strictly decreasing on $(0,\infty)$ and, on $[1,\infty)$, equals $\frac{G(1)}{r}$. Using the positive semidefiniteness of $\theta\theta^T$ and $I-\theta\theta^T$, one obtains
\begin{equation}
    \label{eq:BN}
    B\leq \mu_1(\B)\frac{\int_{\B}G^2 \dd x}{d}I - G(1)^2Z.
\end{equation}
The reciprocal-sum isoperimetric inequality now follows from \cref{eq:recsump,eq:AM,eq:BN,lem:linalg}. If equality holds, the equality characterization gives $A=\frac{\int_{\B}G(r)^2\dd x}{d}I$, namely,
\[
    \int_{\Omega} G(r)^2\theta\theta^T \dd x = \int_{\B}G(r)^2\theta\theta^T \dd x.
\]
Taking traces on both sides yields
\[
    \int_{\Omega}G(r)^2\dd x = \int_{\B}G(r)^2\dd x.
\]
Since $|\Omega|=|\B|$  and   $G$ is nondecreasing on $(0,\infty)$ and strictly increasing on $(0,1)$, the preceding identity implies that $\Omega = \B$. This completes the proof of the theorem.
\end{proof}

Set
\begin{equation}
    \label{eq:notation}
    a=\frac{1}{d}\int_{\B}G^2\,\dd x,\quad b=G(1)^2, \quad \lambda = \mu_1(\B), \quad M=aI+bZ, \quad N=\lambda aI-bZ,
\end{equation}
where $Z$ is as in \cref{eq:defZ}. After rotating the coordinate system, we may assume that $Z$ is diagonal, say
\begin{equation}
    \label{eq:diagassumption}
    Z=\diag(z_1,\dots,z_d), \qquad \sum_{i=1}^d z_i=0.
\end{equation}
Throughout the remainder of the paper, we use the notation and assumptions in \cref{eq:notation,eq:diagassumption}.

The linear-algebraic lemma above admits the following quantitative refinement.
\begin{lemma}
    \label{lem:linalg2}
    Under the assumptions of \cref{lem:linalg}, suppose in addition that
    $\widehat{Z}\geq-c_0I$ for some $c_0>0$. Then
    \begin{equation*}
        \Tr(B^{-1}A) - \frac{d}{\lambda} \geq \frac{\lambda+1}{\lambda^2(\lambda a+c_0)a} \|\widehat{Z}\|_{\text{HS}}^2.
    \end{equation*}
\end{lemma}

\begin{proof}
    Set $M=aI+\widehat{Z}$, $N=\lambda aI-\widehat{Z}$. Since $0< B\leq N$, we have $B^{-1}\geq N^{-1}>0$. Therefore,
    \begin{equation}
        \label{eq:linalgBinvA}
        \Tr(B^{-1}A) = \Tr(B^{-\frac{1}{2}}AB^{-\frac{1}{2}})\geq \Tr(N^{-1}A)\geq \Tr(N^{-1}M).
    \end{equation}
    After an orthogonal change of basis, we may assume that $\widehat{Z}$ is diagonal, say $\diag(\hat{z}_1,\dots,\hat{z}_d)$. Then $\lambda a-\hat{z}_i>0$ for every $i$, and
    \begin{equation}
        \label{eq:NinvMequiv}
        \Tr(N^{-1}M) = \sum_{i=1}^d \frac{a+\hat{z}_i}{\lambda a-\hat{z}_i} = \frac{d}{\lambda}+ \frac{\lambda+1}{\lambda}\sum_{i=1}^{d}\frac{\hat{z}_i}{\lambda a-\hat{z}_i}.
    \end{equation}
    Since $\Tr \widehat{Z}=0$, we have
    \[
        \sum_{i=1}^d \frac{\hat{z}_i}{\lambda a-\hat{z}_i} = \frac{1}{\lambda a} \sum_{i=1}^d \frac{\hat{z}_i^2}{\lambda a-\hat{z}_i}.
    \]
    Hence \cref{eq:NinvMequiv} becomes
    \begin{equation}
        \label{eq:NinvMequiv2}
        \Tr(N^{-1}M) = \frac{d}{\lambda}+ \frac{\lambda+1}{\lambda^2 a}\sum_{i=1}^{d}\frac{\hat{z}_i^2}{\lambda a-\hat{z}_i}.
    \end{equation}
    Since $\widehat{Z}\geq-c_0I$, we have $\lambda a-\hat{z}_i\leq\lambda a+c_0$. Combining this inequality with \cref{eq:linalgBinvA,eq:NinvMequiv2} proves the result.
\end{proof}

\begin{remark}
    This lemma can also be proved using Taylor's theorem with the Lagrange form of the remainder.
\end{remark}

The following estimate is the key quantitative ingredient in our stability arguments.

\begin{lemma}
    \label{lem:tech}
    The matrix $Z$ defined in \cref{eq:defZ} satisfies
    \begin{equation}
        \label{eq:zilowerbd}
        -\frac{\omega_d}{d}I\leq Z< \frac{\lambda a}{b}I
        % -\frac{\omega_d}{d} \leq z_i < \frac{\lambda a}{b}, \qquad i=1,\dots,d.
    \end{equation}
    Furthermore, 
    \begin{equation}
        \label{eq:Dlower}
        \D \geq c_3 \sum_{i=1}^d z_i^2 \geq c_3(\max_{1\leq i\leq d}|z_i|)^2, 
    \end{equation}
    where
    \begin{equation}
        \label{eq:C1}
        c_3=\frac{b^2(\lambda+1)}{\lambda^2 a\left(\lambda a+\frac{\omega_d b}{d}\right)}.
    \end{equation}
\end{lemma}

\begin{proof}
    The upper bound for $Z$ follows from \cref{eq:BN} and positive definiteness of $B$. On the other hand, the definition of $Z$ and the positive semidefiniteness of $\theta\theta^T$ give
    \begin{equation*}
        Z\geq -\int_{\B}\theta\theta^T\,\dd x = -\frac{\omega_d}{d}I.
    \end{equation*}
    This proves \cref{eq:zilowerbd}. 
    
    Apply \cref{lem:linalg2} with $\widehat{Z}=bZ$, $c_0=\frac{\omega_db}{d}$. By \cref{eq:recsump} and the definition of $\D$,
    \[
        \D \geq \Tr(B^{-1}A)-\frac{d}{\lambda} \geq c_3\|Z\|_{\mathrm{HS}}^2 = c_3\sum_{i=1}^d z_i^2.
    \] 
The inequality    \cref{eq:Dlower} follows.
\end{proof}

\section{Stability quantified by the eigenvalues}

In this section, we prove the spectral stability estimates stated in \cref{thm:recsummui,cor:recsumspecgap,thm:sumstability}. More precisely, the reciprocal-sum deficit controls the individual eigenvalue displacements and the splitting of the first eigenvalue cluster quadratically, and the displacement of the sum of the eigenvalues in the cluster linearly. We adopt the notation in \cref{eq:notation} and the assumption \cref{eq:diagassumption}.

\subsection{Individual eigenvalues and spectral splitting}

We begin with the proof of \cref{thm:recsummui}. By scaling invariance, we may assume that $|\Omega|=\omega_d$. Set
\begin{equation}
    \label{eq:epsilon1}
    \varepsilon_0 := \min\left\{ 1,\, \frac{a^2c_3}{4b^2} \right\}.
\end{equation}
We first consider the case $\D\leq \varepsilon_0$. It follows from \cref{eq:Dlower} and the choice of $\varepsilon_0$ that
\begin{equation}
    \label{eq:ziabbound}
    \max_{i}|z_i| \leq \frac{a}{2b}.
\end{equation}
Since $p_d$ is the largest eigenvalue of $A^{-1}B$, the inequalities $A\geq M$ and $B\leq N$ yield
\begin{equation}
    \label{eq:pdvar}
    p_d = \max_{|v|=1} \frac{v^T Bv}{v^T Av} \leq \max_{|v|=1} \frac{v^T N v}{v^T M v} = \max_{|v|=1}\frac{\lambda a-b\zeta}{a+b\zeta},
\end{equation}
where $\zeta=v^TZv$. For every unit vector $v$, we have
\begin{equation}
    \label{eq:abss}
    |\zeta|\leq \max_{i}|z_i| \leq \frac{a}{2b}.
\end{equation}
Consequently, $a+b\zeta\geq a/2$. Combining this with \cref{eq:Dlower,eq:pdvar,eq:abss}, we obtain
\[
    p_d-\lambda \leq \max_{|v|=1}\frac{-b(\lambda+1)\zeta}{a+b\zeta} \leq C_4 \D^{\frac{1}{2}},
\]
where
\[
    C_4= 2\lambda(\lambda+1)^{\frac{1}{2}}\left(\lambda +\frac{\omega_d b}{da}\right)^{\frac{1}{2}} .%  \leq 2\lambda\left(\lambda +\frac{\omega_d b}{da}\right).
\]
Thanks to \cref{eq:mup}, it follows that
\begin{equation}
    \label{eq:mudupper}
    \mu_d(\Omega) \leq \lambda+C_4\D^{\frac{1}{2}}.
\end{equation}

We next derive a lower bound for $\mu_1(\Omega)$. For $i=2,\dots,d$, we have
\[
    \mu_i(\Omega)\leq \mu_d(\Omega) \leq \lambda+C_4\D^{\frac{1}{2}},
\]
and hence
\[
    \frac{d}{\lambda}+\D = \sum_{i=1}^d\frac{1}{\mu_i(\Omega)} \geq \frac{1}{\mu_1(\Omega)} + \frac{d-1}{\lambda+C_4\D^{\frac{1}{2}}}.
\]
Therefore,
\[
\begin{aligned}
    \frac{1}{\mu_1(\Omega)}   \leq \frac{d}{\lambda}+\D - \frac{d-1}{\lambda+C_4\D^{\frac{1}{2}}}        = \frac{1}{\lambda} + \D + \frac{(d-1)C_4\D^{\frac{1}{2}}} {\lambda\big(\lambda+C_4\D^{\frac{1}{2}}\big)}.
\end{aligned}
\]
By the choice of $\varepsilon_0$, we have $\D\leq 1$, and thus
\[
    \frac{1}{\mu_1(\Omega)} \leq \frac{1}{\lambda}+C_5\D^{\frac{1}{2}},
\]
where
\[
    C_5 = 1+\frac{(d-1)C_4}{\lambda^2}.
\]
This yields
\begin{equation}
    \label{eq:mu1lower}
    \mu_1(\Omega) \geq \frac{\lambda}{1+\lambda C_5\D^{\frac{1}{2}}} \geq \lambda-\lambda^2 C_5\D^{\frac{1}{2}}.
\end{equation}
Since $\lambda^2C_5\geq C_4$, combining \cref{eq:mudupper,eq:mu1lower} gives
\begin{equation}
    \label{eq:Deltasmall}
    |\mu_i(\Omega)-\mu_i(\B)|\leq \lambda^2 C_5 \D^{\frac{1}{2}},\qquad i=1,\dots,d.
\end{equation}

It remains to consider the case $\D>\varepsilon_0$. For $d\geq3$, Kr\"oger proved the universal upper bound below in \cite{Kro92}, whereas the two-dimensional sharp estimate is due to Bucur and Henrot \cite{BH19}:
\begin{equation}
    \label{eq:Krogerup}
    \mu_d(\Omega) \leq K_d :=
    \begin{cases}
        4\pi^2 \left( \frac{d(d+2)}{2\omega_d^2} \right)^{\frac{2}{d}}, &d\geq 3,\\
        2\mu_1(\B), & d=2.
    \end{cases}
\end{equation}
Consequently,
\begin{equation}
    \label{eq:Deltanotsmall}
    |\mu_i(\Omega)-\mu_i(\B)| \leq K_d+\mu_1(\B)\leq \varepsilon_0^{-\frac{1}{2}}(K_d+\mu_1(\B))\D^{\frac{1}{2}}.
\end{equation}
Combining \cref{eq:Deltasmall,eq:Deltanotsmall} and setting
\begin{equation}
    \label{eq:gammad}
    \gamma_d = \min\left\{\frac{1}{\mu_1(\B)^4 C_5^2}, \frac{\varepsilon_0}{(K_d+\mu_1(\B))^2} \right\},
\end{equation}
% \[
%     \gamma_d = \min\left\{\frac{1}{(2d-1)^2\lambda^2 \left(\lambda +\frac{\omega_d b}{da}\right)^2}, \frac{1}{(K_d+\mu_1(\B))^2}, \frac{ab^2(\lambda+1)}{(K_d+\mu_1(\B))^2\lambda^2 \left(\lambda a+\frac{\omega_d b}{d}\right)} \right\}
% \]
we complete the proof of \cref{thm:recsummui}.

The spectral splitting estimate in \cref{cor:recsumspecgap} now follows immediately from the triangle inequality
\[
    \big(\muu_d(\Omega)-\muu_1(\Omega)\big)^2 \leq 2\big(\muu_1(\Omega)-\muu_1(\B)\big)^2 + 2\big(\muu_d(\Omega)-\muu_d(\B)\big)^2 \leq \frac{4}{\gamma_d} \D.
\]

\subsection{The sum of the first \texorpdfstring{$d$}{d} eigenvalues}

We now prove \cref{thm:sumstability}. By scaling invariance, we may assume that $|\Omega|=\omega_d$. 

We first assume that $\s\leq \mathcal{S}(\B) $. By the arithmetic--harmonic mean inequality,
\[
    \frac{1}{d}\sum_{i=1}^d\frac{1}{\mu_i(\Omega)} \geq \left( \frac{1}{d}\sum_{i=1}^d\mu_i(\Omega) \right)^{-1} = \frac{d}{\s}.
\]
It follows that
\[
    \D \geq \frac{d^2}{\s} - \frac{d}{\mu_1(\B)} = \frac{d (\mathcal{S}(\B)-\s)}{\mu_1(\B) \s}.
\]
Thus,
\begin{equation}
    \label{eq:negativeside}
    \mathcal{S}(\B)-\s \leq \mu_1(\B)^2\D.
\end{equation}

Let $\varepsilon_0$ be as in \cref{eq:epsilon1}. We next assume that $\s> \mathcal{S}(\B) $ and $\D\leq \varepsilon_0$. By \cref{eq:mup}, we have
\[
    \s- \mathcal{S}(\B)  \leq \sum_{i=1}^d p_i- \mu_1(\B) d  = \Tr (A^{-1}B)- \mu_1(\B) d .
\]
By \cref{eq:AM,eq:ziabbound}, we have $A\geq M\geq \frac{a}{2}I>0$. Together with \cref{eq:BN}, this gives
\[
    \s- \mathcal{S}(\B)  \leq \Tr (M^{-1}N)- \mu_1(\B) d= -b(\mu_1(\B)+1) \sum_{i=1}^d \frac{z_i}{a+bz_i}.
\]
Since $\sum_i z_i=0$ and $M\geq \frac{a}{2}I$, \cref{eq:Dlower} yields
\[
    -\sum_{i=1}^d \frac{z_i}{a+bz_i} = \frac{b}{a} \sum_{i=1}^d \frac{z_i^2}{a+bz_i} \leq \frac{2b}{a^2} \sum_{i=1}^d z_i^2 \leq \frac{2b}{c_3 a^2}\D.
\]
Hence,
\begin{equation}
    \label{eq:positiveside}
    \s- \mathcal{S}(\B)  \leq \frac{2b^2(\mu_1(\B)+1)}{c_3 a^2}\D.
\end{equation}

It remains to treat the case $\D>\varepsilon_0$. By \cref{eq:Krogerup},
\begin{equation}
    \label{eq:largedeficit}
    |\s- \mathcal{S}(\B) | \leq d(K_d+\mu_1(\B)) \leq d\varepsilon_0^{-1}(K_d+\mu_1(\B))\D.
\end{equation}

Combining \cref{eq:negativeside,eq:positiveside,eq:largedeficit}, we obtain
\[
    \D\geq \tau_d |\s- \mathcal{S}(\B) |,
\]
where
\begin{equation}
    \label{eq:taud}
    \tau_d=\min\left\{\frac{1}{\mu_1(\B)^2}, \frac{c_3 a^2}{2b^2 (\mu_1(\B)+1)}, \frac{\varepsilon_0}{d(K_d+\mu_1(\B))}\right\}.
\end{equation}
% Equivalently, after substituting the definitions of $c_3$ and $\varepsilon_0$,
% \[
%     \tau_d=\min\left\{\frac{1}{\lambda^2}, \frac{1}{2\lambda^2 \left(\lambda +\frac{\omega_d b}{da}\right)}, \frac{1}{d(K_d+\lambda)}, \frac{\lambda+1}{4d\lambda^2(K_d+\lambda) \left(\lambda +\frac{\omega_d b}{da}\right)} \right\}.
% \]
This completes the proof of \cref{thm:sumstability}.

\section{Stability quantified by the Fraenkel asymmetry}

In this section, we prove \cref{thm:recsumfraenkel} by retaining the positive-semidefinite remainder in the comparison for the energy matrix and estimating it through a mass-transfer argument.

By scaling invariance, we may assume that $|\Omega|=\omega_d$. We again adopt the notation in \cref{eq:notation} and the assumption \cref{eq:diagassumption}. It follows from \cref{eq:Bexp} that
\[
    B=N-Q,
\]
where
\begin{equation}
    \label{eq:Q}
    Q = \int_{\B\backslash\Omega} G'(r)^2\theta\theta^T \dd x + G(1)^2\int_{\Omega\backslash \B} (1-\frac{1}{r^2})(I-\theta\theta^T) \dd x + \int_{\B\backslash\Omega} (\frac{G(r)^2}{r^2}-G(1)^2)(I-\theta\theta^T) \dd x.
\end{equation}
% The mass matrix similarly satisfies
% \[
%     A=M+P,
% \]
% where
% \[
%     P = \int_{\B\backslash\Omega} (G(1)^2-G(r)^2)\theta\theta^T \dd x.
% \]
Since $B>0$ and $Q=N-B\geq 0$, we have $N>0$. Moreover,
\[
    N^{\frac{1}{2}} \left(I-N^{-\frac{1}{2}}QN^{-\frac{1}{2}}\right) N^{\frac{1}{2}}=N-Q=B>0,
\]
and hence $N^{-\frac{1}{2}}QN^{-\frac{1}{2}}<I$. Combining this with the positive semidefiniteness of $N^{-\frac{1}{2}}QN^{-\frac{1}{2}}$, we obtain
\begin{equation}
\label{eq:Binverselower}
\begin{aligned}
    B^{-1} %& = N^{-\frac{1}{2}} \left(I-N^{-\frac{1}{2}}QN^{-\frac{1}{2}}\right)^{-1} N^{-\frac{1}{2}}\\
      \geq N^{-\frac{1}{2}} \left(I+N^{-\frac{1}{2}}QN^{-\frac{1}{2}}\right) N^{-\frac{1}{2}}  = N^{-1}+N^{-1}QN^{-1}.
\end{aligned}
\end{equation}

We distinguish two cases. First suppose that $\max_{i} |z_i| \geq \frac{a}{2b}$. Since the Fraenkel asymmetry satisfies $\A\leq 2$, \cref{eq:Dlower} yields
\[
    \D \geq \frac{c_3a^2}{4b^2} \geq \frac{c_3a^2}{16b^2}\A^2.
\]

It remains to consider the case $\max_{i} |z_i| < \frac{a}{2b}$. Then
\begin{equation}
    \label{eq:Alower}
    A\geq M=aI+bZ\geq \frac{a}{2} I.
\end{equation}
By \cref{eq:recsump,eq:Binverselower,eq:Alower}, together with the inequality $\Tr(N^{-1}M)\geq \frac{d}{\lambda}$ from \cref{eq:NinvMequiv2}, we obtain
\begin{equation}
    \label{eq:BinvAlower}
    \D\geq \Tr (B^{-1}A)-\frac{d}{\lambda} \geq \Tr (N^{-1}M)-\frac{d}{\lambda} + \Tr (N^{-1}QN^{-1}A) \geq \frac{a}{2}\Tr (N^{-1}QN^{-1}).
\end{equation}
The inequality $\max_{i} |z_i| < \frac{a}{2b}$ gives
\[
    N=\lambda aI-bZ \leq \left(\lambda+1\right) aI.
\]
Consequently,
\[
    N^{-1}\geq \frac{1}{(\lambda+1)a} I.
\]
Substituting this estimate into \cref{eq:BinvAlower} gives
\begin{equation}
    \label{eq:BinvAlowerQ}
    \D \geq \frac{a}{2}\Tr (QN^{-2}) \geq \frac{1}{2(\lambda+1)^2a}\Tr Q.
\end{equation}

We next derive a lower bound for $\Tr Q$.  The first two terms in the definition of $Q$ are positive semidefinite. Indeed, this is immediate for the first term, while the coefficient in the second term is nonnegative on $\Omega\backslash\B$. Taking traces in \cref{eq:Q} and discarding the first two nonnegative terms, we obtain
\begin{equation}
    \label{eq:TrQ}
    \Tr Q\geq (d-1)\int_{\B\setminus\Omega} \left(\frac{G(r)^2}{r^2}-G(1)^2\right)\,\dd x.
\end{equation}
The required lower bound for the integrand is provided by the following lemma.

\begin{lemma}
    \label{lem:radialremainder}
    On $(0,1)$,
    \begin{equation}
        \label{eq:Gr-G1}
        \frac{G(r)^2}{r^2}-G(1)^2\geq J_\nu(p_{\nu,1})^2(1-r).
    \end{equation}
\end{lemma}

Combining \cref{eq:BinvAlowerQ,eq:TrQ,eq:Gr-G1}, we find that
\begin{equation}
    \label{eq:D1-r}
    \D \geq \frac{(d-1) J_\nu(p_{\nu,1})^2}{2(\lambda+1)^2a} \int_{\B\setminus\Omega}(1-r)\,\dd x.
\end{equation}
Let $B_{R_0}$ be the ball centered at the origin with radius $R_0 = \big(\frac{|\B\cap \Omega|}{\omega_d}\big)^{1/d}$. Then $|\B\backslash B_{R_0}|=|\B\backslash \Omega|$, and hence $|B_{R_0}\backslash\Omega|=|(\B\backslash B_{R_0})\cap\Omega|$. Together with the monotonicity of $1-r$, this yields
\begin{equation}
\label{eq:massdispla}
\begin{aligned}
    \int_{\B\backslash\Omega} (1-r)\,\dd x & = \int_{B_{R_0}\backslash\Omega}(1-r)\,\dd x + \int_{(\B\backslash B_{R_0})\backslash\Omega}(1-r)\,\dd x \\
    & \geq \int_{(\B\backslash B_{R_0})\cap\Omega}(1-r)\,\dd x + \int_{(\B\backslash B_{R_0})\backslash\Omega}(1-r)\,\dd x \\
    & = \int_{\B\backslash B_{R_0}} (1-r)\,\dd x.
\end{aligned}
\end{equation}
Passing to polar coordinates, we obtain
\[
    \int_{\B\backslash B_{R_0}} (1-r)\,\dd x = d\omega_d \int_{R_0}^{1} (1-r)r^{d-1} \,\dd r = \omega_d \int_{0}^{1-R_0^d}\big(1-(1-t)^{\frac{1}{d}}\big) \,\dd t.
\]
Since the map $s\mapsto s^{\frac{1}{d}}$ is concave, we have $(1-t)^{\frac{1}{d}}\leq 1-\frac{t}{d}$ on $(0,1)$. Moreover, the definition of the Fraenkel asymmetry gives $|\B\backslash\Omega|\geq\frac{\omega_d}{2}\A$. It follows that
\[
    \int_{\B\backslash B_{R_0}} (1-r)\,\dd x \geq \frac{\omega_d}{2d}(1-R_0^d)^2 = \frac{1}{2d\omega_d}|\B\backslash\Omega|^2\geq \frac{\omega_d}{8d}\A^2.
\]
Combining this inequality with \cref{eq:D1-r,eq:massdispla}, we arrive at
\[
    \D \geq \frac{(d-1)\omega_d J_{\nu}(p_{\nu,1})^2}{16d(\lambda+1)^2a}\A^2.
\]
Setting
\begin{equation}
    \label{eq:kappad}
    \kappa_d = \min\left\{\frac{c_3a^2}{16b^2}, \frac{(d-1)\omega_d J_{\nu}(p_{\nu,1})^2}{16d(\mu_1(\B)+1)^2a}\right\},
\end{equation}
we obtain \cref{eq:FraenkelStability}. 

It remains to prove the radial estimate.

\begin{proof}[Proof of \cref{lem:radialremainder}]
    Let $j_{\nu,k}$ denote the positive zeros of $J_\nu$. The Hadamard product representation
    \begin{equation}
        \label{eq:Hadamardproduct}
        J_{\nu}(s)=\frac{(s/2)^\nu}{\Gamma(\nu+1)}\prod_{k=1}^{\infty}\left(1-\frac{s^2}{j_{\nu,k}^2}\right)
    \end{equation}
    gives
    \[
        \frac{g(r)}{r}=\frac{p_{\nu,1}^\nu}{2^\nu\Gamma(\nu+1)}\prod_{k=1}^{\infty}\left(1-\frac{p_{\nu,1}^2r^2}{j_{\nu,k}^2}\right).
    \]
    The function $s\mapsto s^{1-\nu}J_\nu(s)$ extends continuously to $s=0$, vanishes at both $0$ and $j_{\nu,1}$, and is positive on $(0,j_{\nu,1})$. Hence, %as the first positive   zero of the derivative,  
    \[
        0<p_{\nu,1}<j_{\nu,1}.
    \]
    Consequently,
    \[
        j_{\nu,k}^2-p_{\nu,1}^2>0,
        \qquad k\geq 1.
    \]
    Therefore,
    \begin{equation}
        \label{eq:grproduct}
        \frac{g(r)}{rg(1)} = \prod_{k=1}^{\infty} \left( 1+\frac{p_{\nu,1}^2(1-r^2)}{j_{\nu,k}^2-p_{\nu,1}^2} \right) \geq 1+(1-r^2)\sum_{k=1}^{\infty} \frac{p_{\nu,1}^2}{j_{\nu,k}^2-p_{\nu,1}^2}.
    \end{equation}
    By the standard Bessel recurrence relation,
    \[
        \frac{\dd}{\dd s}\left(s^{1-\nu}J_\nu(s)\right) = s^{-\nu}\left(J_\nu(s)-sJ_{\nu+1}(s)\right).
    \]
    Thus, by the definition of $p_{\nu,1}$,
    \begin{equation}
        \label{eq:JnuJnu1}
        J_\nu(p_{\nu,1})=p_{\nu,1}J_{\nu+1}(p_{\nu,1}).
    \end{equation}
    On the other hand, logarithmic differentiation of \cref{eq:Hadamardproduct} yields
    \[
        \frac{J_{\nu}'(s)}{J_{\nu}(s)} = \frac{\nu}{s} - \sum_{k=1}^{\infty}\frac{2s}{j_{\nu,k}^2-s^2}.
    \]
    Combining this identity with the recurrence relation $J_{\nu}'(s)=\frac{\nu}{s}J_{\nu}(s) - J_{\nu+1}(s)$ gives
    \[
        \frac{J_{\nu+1}(s)}{J_\nu(s)} = 2s\sum_{k=1}^{\infty} \frac{1}{j_{\nu,k}^{2}-s^{2}}.
    \]
    Evaluating this identity at $s=p_{\nu,1}$ and using \cref{eq:JnuJnu1}, we obtain
    \begin{equation}
        \label{eq:pjsum1}
        1 = p_{\nu,1}\frac{J_{\nu+1}(p_{\nu,1})}{J_\nu(p_{\nu,1})} = 2\sum_{k=1}^{\infty} \frac{p_{\nu,1}^2}{j_{\nu,k}^2-p_{\nu,1}^2}.
    \end{equation}
    Consequently, \cref{eq:grproduct,eq:pjsum1} imply that
    \[
        \frac{g(r)}{rg(1)} \geq 1+\frac{1-r^2}{2}.
    \]
    It follows that
    \[
        \frac{g(r)^2}{r^2}-g(1)^2 \geq g(1)^2\left[\left(1+\frac{1-r^2}{2}\right)^2-1\right] \geq g(1)^2(1-r^2) \geq g(1)^2(1-r).
    \]
    Since $g(1)=J_\nu(p_{\nu,1})$, the lemma follows.
\end{proof}

\section{Spectral-gap estimates from the Szeg\H{o}--Weinberger deficit}

In this section, we estimate the splitting of the first nonzero Neumann eigenvalue cluster directly in terms of the Szeg\H{o}--Weinberger deficit. The first result controls the gap between the first two nonzero eigenvalues, while the second controls the full width of the cluster. Both follow from \cref{thm:sumstability} combined with universal upper bounds for higher Neumann eigenvalues due to Bucur and Henrot \cite{BH19} and Kr\"oger \cite{Kro92}.

Throughout this section, we write
\[
    \alpha_d=\tau_d\mu_1(\B)^2.
\]
We first prove \cref{thm:mu2}. By scaling invariance, we may assume that $|\Omega|=\omega_d$. If $\mu_1(\Omega)=\mu_1(\B)$, then $\Omega$ is a ball by the equality characterization in the Szeg\H{o}--Weinberger inequality, and the conclusion follows. We may therefore suppose that $0<\mu_1(\Omega)<\mu_1(\B)$. Define
\[
    \rho=\frac{\mu_1(\Omega)}{\mu_1(\B)},\qquad \delta=\frac{\mu_2(\Omega)-\mu_1(\Omega)}{\mu_1(\B)-\mu_1(\Omega)}.
\]
Let $C_1$ be the larger root of
\begin{equation}
    \label{eq:C4}
    (d-1)(1+\alpha_d 2^{\frac{2}{d}} )s^2-\big[2^{\frac{2}{d}} d +d-1+\alpha_d 2^{\frac{2}{d}} \big(d+(d-1) 2^{\frac{2}{d}} \big)\big]s+2^{\frac{2}{d}} d (1+\alpha_d 2^{\frac{2}{d}} )=0.
\end{equation}
Evaluating the polynomial on the left at $2^{\frac{2}{d}}$, $\frac{d}{d-1}$, and $\frac{2^{\frac{2}{d}}d}{d-1}$ gives
\begin{equation}
    \label{eq:C4bounds}
    \max\left\{2^{\frac{2}{d}},\frac{d}{d-1}\right\}<C_1<\frac{2^{\frac{2}{d}}d}{d-1}.
\end{equation}

Assume to the contrary that $\delta>C_1$. By \cref{eq:C4bounds}, we have $\delta>\frac{d}{d-1}$. It follows from \cref{thm:sumstability} that
\begin{equation}
    \label{eq:mu2Dlower}
    \D\geq\tau_d\big(\s-\mathcal{S}(\B)\big)\geq\tau_d\big((d-1)\delta-d\big)\big(\mu_1(\B)-\mu_1(\Omega)\big).
\end{equation}
On the other hand, by definition,  
\[
    \D\leq\frac{1}{\mu_1(\Omega)}+\frac{d-1}{\mu_2(\Omega)}-\frac{d}{\mu_1(\B)}.
\]
Since
\[
    \mu_1(\B)-\mu_1(\Omega)=\mu_1(\B)(1-\rho),\qquad \mu_2(\Omega)=\mu_1(\B)\big(\rho+\delta(1-\rho)\big),
\]
combining the preceding inequality with \cref{eq:mu2Dlower} gives
\begin{equation}
    \label{eq:mu2rho}
    \alpha_d\big((d-1)\delta-d\big)\leq\frac{1}{\rho}-\frac{(d-1)(\delta-1)}{\rho+\delta(1-\rho)}.
\end{equation}
The sharp Bucur--Henrot bound \cite{BH19} implies
\[
    \rho+\delta(1-\rho)=\frac{\mu_2(\Omega)}{\mu_1(\B)}\leq2^{\frac{2}{d}},
\]
and hence, since $\delta>C_1>2^{\frac{2}{d}}$,
\[
    \rho\geq\frac{\delta-2^{\frac{2}{d}}}{\delta-1}.
\]
The right-hand side of \cref{eq:mu2rho} is decreasing in $\rho$. We therefore obtain
\[
    \alpha_d\big((d-1)\delta-d\big)\leq\frac{(\delta-1)\big(2^{\frac{2}{d}}d-(d-1)\delta\big)}{2^{\frac{2}{d}}\big(\delta-2^{\frac{2}{d}}\big)}.
\]
This becomes
\[
    (d-1)(1+\alpha_d 2^{\frac{2}{d}} )\delta^2-\big[2^{\frac{2}{d}} d +d-1+\alpha_d 2^{\frac{2}{d}} \big(d+(d-1) 2^{\frac{2}{d}} \big)\big]\delta+ 2^{\frac{2}{d}} d (1+\alpha_d 2^{\frac{2}{d}} )\leq0.
\]
This is impossible because $C_1$ is the larger root of the polynomial on the left and $\delta>C_1$. Hence $\delta\leq C_1$, which proves \cref{thm:mu2}.

We next prove \cref{thm:mud}. Assume that $d\geq3$, and write
\[
    L_d=\frac{K_d}{\mu_1(\B)}= \frac{4\pi^2}{\mu_1(\B)} \left(\frac{d(d+2)}{2\omega_d^2}\right)^{\frac{2}{d}}.
\]
Let $C_2$ be the larger root of
\begin{equation}
    \label{eq:C5}
    (1+\alpha_dL_d)s^2-\big[dL_d+1+\alpha_dL_d(d+L_d)\big]s+dL_d(1+\alpha_dL_d)=0.
\end{equation}
As above, evaluating the polynomial on the left at $L_d$, $d$, and $dL_d$ gives
\begin{equation}
    \label{eq:C5bounds}
    \max\{L_d,d\}<C_2<dL_d.
\end{equation}

As in the proof of \cref{thm:mu2}, we normalize $|\Omega|=\omega_d$ and only need to consider the case $0<\mu_1(\Omega)<\mu_1(\B)$. Define
\[
    \rho=\frac{\mu_1(\Omega)}{\mu_1(\B)},\qquad \widetilde{\delta}=\frac{\mu_d(\Omega)-\mu_1(\Omega)}{\mu_1(\B)-\mu_1(\Omega)}.
\]
Assume to the contrary that $\widetilde{\delta}>C_2$. By \cref{eq:C5bounds}, we have $\widetilde{\delta}>d$. The same ordering argument as above, combined with \cref{thm:sumstability}, gives
\[
    \D\geq\tau_d\big(\s-\mathcal{S}(\B)\big)\geq\tau_d(\widetilde{\delta}-d)\big(\mu_1(\B)-\mu_1(\Omega)\big),
\]
while
\[
    \D\leq\frac{d-1}{\mu_1(\Omega)}+\frac{1}{\mu_d(\Omega)}-\frac{d}{\mu_1(\B)}.
\]
Then we can arrive at a contradiction by using exactly the same argument as in the proof of Theorem \ref{thm:mu2}, the only difference being replacing Bucur-Henrot bound by  Kr\"oger's bound in \cref{eq:Krogerup}. This finishes the proof of \cref{thm:mud}.
\iffalse 
Since
\[
    \mu_1(\B)-\mu_1(\Omega)=\mu_1(\B)(1-\rho),\qquad \mu_d(\Omega)=\mu_1(\B)\big(\rho+\widetilde{\delta}(1-\rho)\big),
\]
we obtain
\begin{equation}
    \label{eq:mudrho}
    \alpha_d(\widetilde{\delta}-d)\leq\frac{d-1}{\rho}-\frac{\widetilde{\delta}-1}{\rho+\widetilde{\delta}(1-\rho)}.
\end{equation}
Kr\"oger's bound in \cref{eq:Krogerup} implies
\[
    \rho+\widetilde{\delta}(1-\rho)=\frac{\mu_d(\Omega)}{\mu_1(\B)}\leq L_d,
\]
and hence, since $\widetilde{\delta}>C_2>L_d$,
\[
    \rho\geq\frac{\widetilde{\delta}-L_d}{\widetilde{\delta}-1}.
\]
The right-hand side of \cref{eq:mudrho} is decreasing in $\rho$. Therefore
\[
    \alpha_d(\widetilde{\delta}-d)\leq\frac{(\widetilde{\delta}-1)(dL_d-\widetilde{\delta})}{L_d(\widetilde{\delta}-L_d)}.
\]
Equivalently,
\[
    (1+\alpha_dL_d)\widetilde{\delta}^2-\big[dL_d+1+\alpha_dL_d(d+L_d)\big]\widetilde{\delta}+dL_d(1+\alpha_dL_d)\leq0,
\]
which contradicts the definition of $C_2$. Hence $\widetilde{\delta}\leq C_2$, which proves \cref{thm:mud}.
\fi

\section{Sharpness of the exponents}

In this section, we prove that all the exponents appearing in the stability results established in the preceding sections are optimal.

Let $\phi$ be a smooth function on $\mathbb{S}^{d-1}$ satisfying
\begin{equation}
    \label{eq:phi01sphereharmonics}
    \ints \phi \,\dd \sigma= 0, \qquad \ints \phi \theta\,\dd \sigma =0.
\end{equation}
In a neighborhood of $\mathbb{S}^{d-1}$, define $\vv(x)=\phi(\frac{x}{r})\frac{x}{r}$ and extend it to a smooth vector field in a neighborhood of $\B$. The normal component of $\vv|_{\mathbb{S}^{d-1}}$ is then given by $V_n=\vv\cdot\nn = \phi$. Define $T_t = \text{Id}+t\vv$ and $\Omega_t=T_t(\B)$. Then
\[
    \partial\Omega_t = \{(1+t\phi(\theta))\theta: \theta\in\mathbb{S}^{d-1}\}.
\]
Moreover, the deformation $T_t$ depends analytically on $t$ near $t=0$, and $T_t:\B\to\Omega_t$ is a smooth diffeomorphism for any sufficiently small $|t|$. By the standard transport formula (see, for example, \cite[Chapter 9, \S4.1]{DZ11}),
\begin{equation}
    \label{eq:volder}
    \left.\frac{\dd}{\dd t}\right|_{t=0}|\Omega_t| = \ints \phi \,\dd \sigma= 0.
\end{equation}

After pulling the Neumann problem back to $\B$, standard analytic perturbation theory for boundary deformations (see, for example, \cite[Chapter VII, \S6.5]{Kat76}) yields $d$ analytic eigenvalue branches issuing from the first nonzero eigenvalue of $\B$. Denote these branches by $\mu_i(t)$, $i=1,\dots,d$. Thus,
\begin{equation}
    \label{eq:muit}
    \mu_i(t)=\lambda+\dot{\mu}_i(0)t+O(t^2).
\end{equation}
The branches $\mu_i(t)$ need not be ordered by magnitude. Standard shape derivative computations for multiple Neumann eigenvalues (see, for example, \cite[Proposition 2]{ZHL20}) identify the coefficients $\dot{\mu}_i(0)$ as the eigenvalues of the symmetric matrix $\hh$ with entries 
\[
    \hh_{ij} =  \ints (\nabla_\tau v_i\cdot\nabla_{\tau}v_j-\lambda v_iv_j)V_n \,\dd\sigma,
\]
where $v_1,\dots,v_d$ form an $L^2(\B)$-orthonormal basis of the eigenspace associated with $\lambda$. These eigenfunctions are given by $v_i=\sqrt{\frac{d}{\int_{\B}g^2~\dd x}}g(r)\frac{x_i}{r}$. On $\mathbb{S}^{d-1}$, since $\nabla_\tau\theta_i=e_i-\theta_i\theta$, where $e_i$ is the $i$-th coordinate vector, we have
\[
    \nabla_\tau v_i\cdot\nabla_\tau v_j-\lambda v_iv_j=\frac{dg(1)^2}{\int_{\B}g^2~\dd x}\big(\delta_{ij}-(\lambda+1)\theta_i\theta_j\big).
\]
Substituting this identity into the definition of $\hh$ and using $\ints\phi~\dd\sigma=0$, which eliminates the term containing $\delta_{ij}$, gives
\begin{equation}
    \label{eq:H1}
    \hh=-\frac{d(\lambda+1)g(1)^2}{\int_{\B}g^2~\dd x}\ints\phi\theta\theta^T~\dd\sigma.
\end{equation}
Consequently,
\[
    \sum_{i=1}^d \dot{\mu}_i(0) = \Tr\hh=-\frac{d(\lambda+1)g(1)^2}{\int_{\B}g^2\dd x}\ints \phi \,\dd\sigma = 0. 
\]
In view of \cref{eq:volder}, the scale normalization of the eigenvalues contributes no first-order term. After ordering the branches for $t>0$ and writing $\muu_i(t)=\muu_i(\Omega_t)$, we therefore obtain
\[
    \sum_{i=1}^d \muu_i(t) = d\lambda + O(t^2) 
\]
and
\[
    \sum_{i=1}^d \frac{1}{\muu_i(t)} = \frac{d}{\lambda} - \frac{1}{\lambda^2}\sum_{i=1}^d \dot{\mu}_i(0) t +O(t^2) = \frac{d}{\lambda} + O(t^2).
\]

To prove the optimality of the exponent $2$ in \cref{thm:recsummui,cor:recsumspecgap}, 
we first consider a perturbation that splits the eigenvalue cluster at first order. Choose
\[
    \phi=\frac{(d+2)\int_{\B}g^2~\dd x}{2d\omega_d(\lambda+1)g(1)^2}(\theta_1^2-\theta_d^2).
\]
The standard spherical moment identities
\[
    \ints\theta_i^4~\dd\sigma=\frac{3\omega_d}{d+2},\qquad \ints\theta_i^2\theta_j^2~\dd\sigma=\frac{\omega_d}{d+2}\quad(i\neq j)
\]
hold, while every mixed moment containing an odd power vanishes. It follows that
\[
    \hh=\diag(-1,0,\dots,0,1).
\]
Hence,
\[
    \muu_{1}(t)=\lambda-t+O(t^2), \quad \muu_d(t) = \lambda+t+O(t^2).
\]
Therefore, as $t\to 0+$,
\[
    \mathcal{D}(\Omega_t) = O(t^2), \quad \muu_1(\B)-\muu_1(\Omega_t)=t+ O(t^2) ,\quad \muu_d(\Omega_t)-\muu_1(\Omega_t) = 2t+O(t^2).
\]
Thus, for every exponent $\alpha<2$, as $t\to 0+$,
\[
    \frac{\mathcal{D}(\Omega_t)}{\big(\muu_1(\B)-\muu_1(\Omega_t)\big)^\alpha}\to 0,\quad \frac{\mathcal  D ({\Omega_t})}{\big(\muu_d(\Omega_t)-\muu_1(\Omega_t)\big)^\alpha}\to 0.
\]
These limits prove the optimality of the exponent $2$ in \cref{thm:recsummui,cor:recsumspecgap}. The same perturbation shows that the linear controls in \cref{thm:mu2,thm:mud} are optimal. Indeed, the classical Szeg\H{o}--Weinberger deficit and the relevant spectral gaps are all  comparable to $t$ as $t \to 0+$.

For the remaining sharpness statements, choose a nonzero smooth function $\phi$ that is orthogonal to all spherical harmonics of degrees at most $2$. In particular, it satisfies \cref{eq:phi01sphereharmonics}. The estimate for nearly spherical domains in \cite[Lemma 6.2]{BDR12} gives
\begin{equation}
    \label{eq:Fraenkelsimt}
    \mathcal{A}(\Omega_t)\asymp |t|.
\end{equation}
Moreover, \cref{eq:H1} gives $\hh=0$, and hence $\muu_i(t)=\lambda+O(t^2)$ for every $i=1,\dots,d$. It follows that
\begin{equation}
    \label{eq:Dt2}
    \mathcal{D}(\Omega_t)=O(t^2).
\end{equation}
Combining \cref{eq:FraenkelStability,eq:Fraenkelsimt,eq:Dt2}, we obtain
\[
    \mathcal{D}(\Omega_t)\asymp \mathcal{A}(\Omega_t)^2\asymp t^2.
\]
This proves that the quadratic dependence on the Fraenkel asymmetry is optimal.

The same family also establishes the optimality of the exponent in \cref{eq:sumstability}. By definition,
\[
\begin{aligned}
    \mathcal{D}(\Omega_t) & = \sum_{i=1}^d \left(\frac{1}{\muu_i(t)}-\frac{1}{\lambda}\right) \\
    & = \sum_{i=1}^d \left(-\frac{\muu_i(t)-\lambda}{\lambda^2}+\frac{(\muu_i(t)-\lambda)^2}{\lambda^2\muu_i(t)}\right)\\
    & = -\frac{1}{\lambda^2}\big(\mathcal{S}(\Omega_t)-\mathcal{S}(\B)\big) + \frac{1}{\lambda^2}\sum_{i=1}^d \frac{(\muu_i(t)-\lambda)^2}{\muu_i(t)}\\
    & = -\frac{1}{\lambda^2}\big(\mathcal{S}(\Omega_t)-\mathcal{S}(\B)\big) + O(t^4).
\end{aligned}
\]
Since $\mathcal{D}(\Omega_t)\asymp t^2$, it follows that
\[
    \mathcal{S}(\B)-\mathcal{S}(\Omega_t)\asymp t^2.
\]
Thus the linear control in \cref{eq:sumstability} is optimal.

\section{Improved constraints on the \texorpdfstring{$(\muu_1,\muu_2)$}{(mu1, mu2)}-image}

Inequalities involving several eigenvalues simultaneously may be viewed as constraints on their joint spectral image. Separate isoperimetric inequalities provide coordinatewise constraints, whereas the reciprocal-sum inequality couples the eigenvalues and excludes additional parts of the region allowed by those separate bounds. The matrix argument underlying the reciprocal-sum inequality contains still finer joint information.

%In dimension two, the matrix $Z$ is determined, up to an orthogonal change of coordinates, by a single scalar parameter. Retaining this parameter throughout the argument yields an additional explicit constraint on the pair $(\muu_1,\muu_2)$. Combined with the previously known inequalities, this constraint improves the description of the joint spectral image.

More precisely, consider the normalized spectral image
\[
    \mathfrak{I}_2 := \left\{ \big(\muu_1(\Omega),\muu_2(\Omega)\big): \Omega\subset\mathbb{R}^2 \text{ is a bounded Lipschitz domain} \right\}.
\]
By scaling invariance, we may assume that $|\Omega|=\pi$, so that
$\muu_i(\Omega)=\mu_i(\Omega)$. Set
\[
    x=\muu_1(\Omega),\qquad y=\muu_2(\Omega),\qquad \lambda=\mu_1(\B)\approx3.3899577167.
\]
The ordering of the eigenvalues, the Szeg\H{o}--Weinberger inequality \cite{Wein56}, and the sharp Bucur--Henrot bound \cite{BH19} give
\begin{equation}
    \label{eq:image-basic-bounds}
    0<x\leq y,\qquad x\leq\lambda,\qquad y\leq 2\lambda.
\end{equation}
Moreover, the reciprocal-sum inequality \cref{thm:HLT} reads
\begin{equation}
    \label{eq:image-deficit}
    \D=\frac{1}{x}+\frac{1}{y}-\frac{2}{\lambda}\geq0.
\end{equation}

We now retain the full information contained in the matrix argument. Write $Z=\diag(z,-z)$. After interchanging the two coordinates if necessary, we may assume that $z\geq0$, and we set $t=\frac{bz}{a}$. Since $N>0$, we have $0\leq t<\lambda$. A direct computation gives
\begin{equation}
    \label{eq:image-Ft}
    \D \geq \Tr(N^{-1}M)-\frac{2}{\lambda} =\frac{2(\lambda+1)t^2} {\lambda(\lambda^2-t^2)} =:F(t).
\end{equation}
Notice that $F$ is increasing on $[0,\lambda)$.

Suppose first that $y\geq\lambda$. If $t<1$, then $M>0$, and the variational estimate for the second Ritz value gives
\[
    y\leq p_2 \leq \frac{\lambda+t}{1-t}.
\]
Equivalently,
\begin{equation}
    \label{eq:image-t-lower}
    t\geq \frac{y-\lambda}{y+1}.
\end{equation}
The same conclusion remains valid when $t\geq1$. Indeed,
\cref{eq:image-basic-bounds} implies
\[
    0\leq\frac{y-\lambda}{y+1} \leq\frac{\lambda}{2\lambda+1} <1\leq t.
\]
Thus \cref{eq:image-t-lower} holds in both cases. Combining it with the monotonicity of $F$ in \cref{eq:image-Ft}, we obtain
\begin{equation}
    \label{eq:image-mu2-refined}
    \D\geq F\left(\frac{y-\lambda}{y+1}\right) = \frac{2(y-\lambda)^2} {\lambda y\big((\lambda-1)y+2\lambda\big)}.
\end{equation}

Substituting \cref{eq:image-deficit} into \cref{eq:image-mu2-refined} and clearing the positive denominators shows that
\begin{equation}
    \label{eq:image-bilinear}
    2xy\leq(\lambda-1)(x+y)+2\lambda.
\end{equation}
Although \cref{eq:image-mu2-refined} was derived under the assumption $y\geq\lambda$, the bilinear inequality \cref{eq:image-bilinear} also holds automatically when $y\leq\lambda$. Indeed, for fixed $y\leq\lambda$, the expression
\[
    (\lambda-1)(x+y)+2\lambda-2xy
\]
is affine in $x\in(0,y]$. Its values at $x=0$ and $x=y$ satisfy
\[
    (\lambda-1)y+2\lambda>0, \qquad 2(\lambda-y)(y+1)\geq0.
\]
Consequently, \cref{eq:image-bilinear} holds for every admissible pair $(x,y)$.

Solving \cref{eq:image-bilinear} for $y$ gives
\[
    y\leq \frac{(\lambda-1)x+2\lambda} {2x+1-\lambda}
\]
whenever $x>\frac{\lambda-1}{2}$. Since
\[
    \frac{2\lambda^2}{3\lambda+1}-\frac{\lambda-1}{2}=\frac{(\lambda+1)^2}{2(3\lambda+1)}>0,
\]
the preceding division is legitimate throughout the second branch in \cref{prop:2dimmu2upper}. The equality statement follows by tracing the equality cases in the preceding argument, as in the proof of \cref{thm:HLT}, and by using  the equality characterization in the Bucur--Henrot inequality. This proves \cref{prop:2dimmu2upper}.

Finally, we explain the specific values  appearing in Figure \ref{fig:mu1mu2-image}. 
The graph of the second branch in  \cref{prop:2dimmu2upper} intersects the Bucur--Henrot bound $y=2\lambda$ at
\[
    x_*=\frac{2\lambda^2}{3\lambda+1} \approx2.0576443737.
\]
It follows that the $(\muu_1,\muu_2)$-image is contained in
\begin{equation}
    \label{eq:image-refined-region}
    \left\{ (x,y): 0<x\leq\lambda,~x\leq y\leq U_*(x) \right\},
\end{equation}
where
\begin{equation}
    \label{eq:image-improved-upper-bound}
    U_*(x)=
    \begin{cases}
        2\lambda, & 0<x\leq \dfrac{2\lambda^2}{3\lambda+1},\\[3mm]
        \dfrac{(\lambda-1)x+2\lambda} {2x+1-\lambda}, & \dfrac{2\lambda^2}{3\lambda+1} \leq x\leq\lambda.
    \end{cases}
\end{equation}

For comparison, combining the reciprocal-sum inequality \cref{eq:image-deficit} only with the bound $y\leq2\lambda$ gives
\begin{equation*}
    % \label{eq:image-reciprocal-upper-bound}
    U_{\mathrm{rec}}(x)=
    \begin{cases}
        2\lambda, & 0<x\leq\dfrac{2\lambda}{3},\\[3mm]
        \dfrac{\lambda x}{2x-\lambda}, & \dfrac{2\lambda}{3}\leq x\leq\lambda.
    \end{cases}
\end{equation*}
The relevant points in \cref{fig:mu1mu2-image} are
\[
    P_*= \left( \frac{2\lambda^2}{3\lambda+1},2\lambda \right) \approx(2.0576,6.7799),
\]
\[
    P_0= \left( \frac{2\lambda}{3},2\lambda \right) \approx(2.2600,6.7799),
\]
and
\[
    Q= \left( \frac{2\lambda}{3}, \frac{2\lambda(\lambda+2)}{\lambda+3} \right) \approx(2.2600,5.7189).
\]
 
The preceding constraints exclude a nontrivial part of the region allowed by the previously known inequalities, but they do not characterize which points of the remaining region are realized by domains. This raises the natural problem of describing the spectral image $\mathfrak{I}_2$ precisely.

 \section*{Declaration on the Use of Artificial Intelligence}
 
 During the preparation of this manuscript, the authors used ChatGPT for language polishing and grammar checking. ChatGPT was also used to assist in identifying the explicit constant appearing in Lemma \ref{lem:radialremainder}. The authors independently verified the mathematical derivation and the correctness of this constant, reviewed and edited all AI-assisted content, and take full responsibility for the content of the manuscript.

%%%%%%%%%%%%%%%%%%%%%%%%%%%%%%%%%%%%%%%%%%%%%%%%%%%%%%%%%%%%%%%%%%%%%%%%%%%%%%%%%%%%%%%%%%%%%

%%%%%%%%%%%%%%%%%%%%%%%%%%%%%%%%%%%%%%%%%%%%%%%%%%%%%%%%%%%%%%%%%%%%%%%%%%

%%%%%%%%%%%%%%%%%%%%%%%%%%%%%%%%%%%%%%%%%%%%%%%%%%%%%%%%%%%%%%%%%%%%%%%%%%%%%%%%%%%%%%%%%%%%%%%

% \bibliographystyle{amsplain}
% \bibliography{ref}

\providecommand{\bysame}{\leavevmode\hbox to3em{\hrulefill}\thinspace}
\providecommand{\MR}{\relax\ifhmode\unskip\space\fi MR }
% \MRhref is called by the amsart/book/proc definition of \MR.
\providecommand{\MRhref}[2]{%
  \href{http://www.ams.org/mathscinet-getitem?mr=#1}{#2}
}
\providecommand{\href}[2]{#2}

\end{document}